\documentclass[12pt,a4paper]{article}
\usepackage{amsmath}
\usepackage{amsfonts}
\usepackage{latexsym,amsthm,amscd}
\usepackage[left=2cm,bottom=2.7cm,top=2.7cm,right=2cm]{geometry}
\usepackage[all]{xy}
\usepackage{graphicx}
\newcounter{alphthm}
\newtheorem{thm}{Theorem}[section]
\newtheorem{lem}[thm]{Lemma}
\newtheorem{cor}{Corollary}[section]

\theoremstyle{definition}

\newtheorem{ex}{Example}
\newcommand{\be}{\begin{equation}}
\newcommand{\ee}{\end{equation}}
\newcommand{\pa}{{\partial}}

\newcommand{\e}{{\epsilon}}

\title{On a Class of Generalized Berwald Manifolds}
\author{A. Tayebi and F. Eslami}
\numberwithin{equation}{section}
\begin{document}
\maketitle
\begin{abstract}
The class of generalized Berwald metrics contains the class of Berwald metrics. In this paper, we characterize two-dimensional generalized Berwald  $(\alpha, \beta)$-metrics with vanishing S-curvature. Let $F=\alpha\phi(s)$, $s=\beta/\alpha$,  be  a two-dimensional   generalized Berwald  $(\alpha,\beta)$-metric on a manifold $M$.   Suppose that $F$ has vanishing S-curvature. We show that one of the following holds: (i) if $F$ is a regular metric, then it reduces to a  Riemannian metric of isotropic sectional curvature or a locally Minkowskian metric;  (ii) if $F$ is an almost regular metric that is not Riemannian nor locally Minkowskian, then we find the explicit form of $\phi=\phi(s)$ which obtains a generalized Berwald  metric that is neither a Berwald nor Landsberg nor a Douglas metric. This provides a generalization of Szab\'{o} rigidity theorem for the class of  $(\alpha,\beta)$-metrics. In the following, we prove that left invariant Finsler  surfaces with vanishing S-curvature must be  Riemannian surfaces of constant sectional curvature. Finally, we construct a family of odd-dimensional generalized Berwald Randers metrics.
\end{abstract}

%--------------------------------------------------------------------------------------------------------------------------------------------------
\section{Introduction}
%--------------------------------------------------------------------------------------------------------------------------------------------------
A Finsler metric $F$ on a manifold $M$ is called  a generalized Berwald metric if there exists a covariant derivative  $\nabla$ on $M$ such that the parallel translations induced by $\nabla$ preserve the Finsler function $F$ \cite{SzSz2}\cite{BSZI}\cite{TB}\cite{V1}. In this case, $F$ is called a generalized Berwald metric on $M$ and  $(M, F)$ is called a generalized Berwald  manifold. If the covariant derivative $\nabla$ is also torsion-free, then $F$ reduces to a Berwald metric. Therefore, the class of Berwald metrics belongs to the class of generalized Berwald metrics. The importance of the class of generalized Berwald manifolds lies  in the fact that these manifolds may have a rich isometry group \cite{Szabo}\cite{SzSz1}. For some  progress about the class of  generalized Berwald manifolds, see \cite{A15}, \cite{BM},  \cite{TB}, \cite{Vin0}, \cite{V1},  \cite{Vin2} and  \cite{V3}.

To find generalized Berwald metrics, one can consider the class of $(\alpha, \beta)$-metrics. An $(\alpha, \beta)$-metric is a Finsler metric defined by $F:=\alpha\phi(s)$, $s={\beta}/{\alpha}$, where  $\phi=\phi(s)$ is a smooth function on a symmetric interval $(-b_0, b_0)$ with certain regularity, $\alpha=\sqrt{a_{ij}(x)y^iy^j}$ is a Riemannian metric and  $\beta =b_i(x) y^i$ is a 1-form on the base manifold. The simplest $(\alpha, \beta)$-metrics are the Randers metrics $F=\alpha+\beta$ which  were discovered by G. Randers when he studied 4-dimensional  general relativity. These metrics have been widely applied in many areas of natural sciences, including physics, biology,  psychology, etc \cite{CSb}. In \cite{V1}, Vincze proved that a Randers metric $F=\alpha+\beta$ is a generalized Berwald metric if and only if dual vector field $\beta^\sharp$ is of constant Riemannian length. In \cite{TB}, Tayebi-Barzegari  generalized Vincze's result for  $(\alpha, \beta)$-metrics and showed that an $(\alpha, \beta)$-metric satisfying the so-called sign property is a generalized Berwald metric if and only if $\beta^\sharp$ is of constant Riemannian length. Then, Vincze showed that an $(\alpha, \beta)$-metric satisfying the regularity property $\phi'(0)\neq 0$ is a generalized Berwald metric if and only if $\beta^\sharp$ is of constant Riemannian length \cite{Vin0}.  In \cite{BM}, Bartelme${\ss}$-Matveev proved that a Finsler metric is a generalized Berwald metric if and only if it is monochromatic.

Generally, two-dimensional Finsler metrics have some different and special Riemannian and non-Riemannian curvature properties from the higher dimensions. For example, Bartelme${\ss}$-Matveev showed that except for torus and the Klein bottle, the other closed 2-dimensional manifolds can not have non-Riemannian generalized Berwald metrics \cite{BM}. In \cite{V3}, Vincze et al. showed that a connected generalized Berwald surface is a Landsberg surface if and only if it is a Berwald
surface. The S-curvature is constructed by Shen for given comparison theorems on Finsler manifolds \cite{CSb}. An interesting problem is to study  generalized Berwald metrics with vanishing S-curvature.  Here, we  characterize the class of two-dimensional generalized Berwald  $(\alpha,\beta)$-metrics with vanishing S-curvature and prove the following.

\begin{thm}\label{MainTHM}
Let $F=\alpha\phi(s)$, $s=\beta/\alpha$,  be  a two-dimensional   generalized Berwald  $(\alpha,\beta)$-metric on a connected manifold $M$. Suppose that $F$ has vanishing S-curvature.
Then one of the following holds:
\begin{description}
  \item[] (i) If $F$ is a regular metric, then it reduces to a  Riemannian metric of isotropic sectional curvature or a locally Minkowskian metric;
  \item[] (ii) If $F$ is an almost regular metric that is not Riemannian nor locally Minkowskian, then $\phi$ is given by
\begin{eqnarray}
\phi=c\exp\Bigg[\int_0^s{\frac{kt+q\sqrt{b^2-t^2}}{1+kt^2+qt\sqrt{b^2-t^2}}dt}\Bigg],\label{unix}
\end{eqnarray}
where  $c>0$, $q>0$, and $k$ are real constants, and $\beta$ satisfies
\be
r_{ij}=0, \ \ \ \  s_i=0.
\ee
In this case, $F$ is neither  a Berwald nor Landsberg nor a Douglas metric.
\end{description}
\end{thm}

It is remarkable that, for a generalized Berwald  $(\alpha, \beta)$-metric  $F=\alpha\phi(s)$, $s=\beta/\alpha$, on an $n$-dimensional manifold $M$,  we show that ${\bf S}=0$ if and only if $F$ is Riemannian or  $\beta$ is a Killing form with  constant length (see Lemma \ref{LS}).

The celebrated  Szab\'{o} rigidity theorem states that  every  2-dimensional Berwald   surface is either locally Minkowskian  or Riemannian (see \cite{ShDiff}). Every Berwald metric has vanishing S-curvature \cite{TR}. Then,  Theorem \ref{MainTHM} is an extension of Szab\'{o}'s result for $(\alpha,\beta)$-metrics.

\smallskip

The condition of generalized Berwaldness  can not be dropped  from the assumption of Theorem \ref{MainTHM}. There are many  two-dimensional $(\alpha,\beta)$-metrics with vanishing S-curvature which are not Riemannian nor locally Minkowskian. For example, let us consider  Shen's Fish-Tank  Randers metric $F=\alpha+\beta$ on $\mathbb{R}^2$ given by following
\begin{eqnarray*}
F={\sqrt{ (xv-yu)^2  + (u^2+v^2) (1-x^2-y^2) } \over 1-x^2-y^2}+ {yu-xv\over 1- x^2-y^2 }.
\end{eqnarray*}
$F$ has vanishing S-curvature \cite{ShSK} while it is not a generalized Berwald metric ($\|\beta\|_\alpha = \sqrt{x^2 + y^2}$). This metric is not Riemannian. Also, $F$ has vanishing flag curvature while it is not locally Minkowskian.

In Theorem \ref{MainTHM},  the vanishing of  S-curvature is necessary. For example,  see the following.
\begin{ex}\label{ex2}
Let us consider $G:=\{(x, y)\in \mathbb{R}^2| y>0\}$ and define a multiplication on $G$ by $(x_1, y_1)\ast (x_2, y_2) := (x_2y_1 + x_1, y_1y_2)$, for $(x_i, y_i)\in G$, $i = 1, 2$. $(G, \ast)$ is a Lie group \cite{HM}. One can  introduce  a Randers metric $F$ on $G$ by
\be
F=\frac{\sqrt{2dx^2+2dxdx+2dy^2}}{y}+\frac{dx+dy}{y}.\label{MH}
\ee
Then
\begin{eqnarray*}
&&a_{11}=a_{22}=\frac{2}{y^2},\ \  \ a_{21}=a_{12}=\frac{1}{y^2},  \  \ \  b_1=b_2=\frac{1}{y},\\
&& a^{11}=a^{22}=\frac{2}{3}y^2, \ \ \ a^{12}=a^{21}=-\frac{1}{3}y^2, \  \ \  b^1=b^2=\frac{1}{3}y.
\end{eqnarray*}
It is easy to see that $\alpha$ is a positive-definite Riemannian metric. Also, we get
\[
b:=\|\beta\|_{\alpha}=\sqrt{a^{ij}b_ib_j}=\sqrt{b_ib^i}=\sqrt{b_1b^1+b_2b^2}=\sqrt{\frac{2}{3}}
\]
It follows that $F$ is a positive-definite  Randers metric on $G$. In \cite{V1}, Vincze showed  that a Randers metric $F=\alpha+\beta$ is a generalized Berwald metric if and only if $\beta$ is of constant length with respect to $\alpha$. Then, the Randers metric  defined by \eqref{MH} is a generalized  Berwald metric. We have
\[
d\beta=\frac{1}{y^2}dx\wedge dy\neq0.
\]
Thus $\beta$ is not closed which implies that $F$ can not be a Berwald metric. We have
\be
s_{11}=s_{22}=0, \ \ \  s_{12}=\frac{1}{y^2}=-s_{21}, \ \ \ s_1=-\frac{1}{3y}, \ \ \  s_2=\frac{1}{3y}.\label{ss}
\ee
where $s_{ij}$ and $s_i$ defined by \eqref{rs1} and \eqref{rs2}. We claim that $F$ has not vanishing S-curvature. On contrary, let ${\bf S}=0$. Then by Lemma 3.2.1 in \cite{CSb},  we have $r_{ij}+b_is_j+b_js_i=0$. Contracting it with $b^i$ yields
\[
r_j+b^2s_j=0.
\]
By considering  $r_i+s_i=0$ and $b^2=2/3$, we must have $s_i=0$ which  contradicts with \eqref{ss}. Then, the Randers metric \eqref{MH} is a generalized Berwald metric with ${\bf B}\neq 0$, ${\bf L}\neq 0$,  ${\bf D}\neq 0$, ${\bf S}\neq 0$ and ${\bf E}\neq 0$. We claim that $F$ is not R-quadratic. On the contrary, let the Finsler metric $F$ defined by \eqref{MH} be  R-quadratic. By  Theorem 6.2.1 in \cite{CSb}, if the two-dimensional Randers metric \eqref{MH} is R-quadratic then it has constant S-curvature ${\bf S}=3cF$, for some real constant  $c$. In \cite{XDg}, Xu-Deng proved that every homogeneous Finsler metric of isotropic S-curvature has vanishing S-curvature. Thus, we must have ${\bf S}=0$ which is a contradiction. Therefore, the Randers metric \eqref{MH}  is not R-quadratic.
\end{ex}

Theorem \ref{MainTHM} may not be held for an $(\alpha,\beta)$-metric of  constant S-curvature. For example, at a point $x=(x^1, x^2)\in \mathbb{R}^2$ and in the direction $y=(y^1, y^2)\in T_x\mathbb{R}^2$, consider the Riemannian metric $\alpha=\alpha(x,y)$ and one form $\beta=\beta(x, y)$ as follows
\begin{eqnarray}
\alpha(x,y):=\sqrt{(y^1)^2+e^{2x^1}(y^2)^2}, \ \ \ \beta(x, y):=y^1.\label{mw}
\end{eqnarray}
Then $s_{ij}=0$ and $r_{ij}=b^2a_{ij}-b_ib_j$ which yield $r_i+s_i=0$. If $\phi=\phi(s)$ satisfies
\begin{eqnarray}
\Phi=-6k\frac{\phi\Delta^2}{b^2-s^2},\label{P}
\end{eqnarray}
for some constant $k$, then $F=\alpha\phi(\beta/\alpha)$ has constant S-curvature  ${\bf S}=3kF$ (see \cite{ChSh3}). Here, $\Delta=\Delta(b, s)$ and $\Phi=\Phi(b, s)$ defined by \eqref{del} and \eqref{phi}, respectively.  The  existence of regular solution of \eqref{P} for
arbitrary $k\in \mathbb{R}$, when $\alpha$ and $\beta$ are given by \eqref{mw}, is proved  in \cite{MW}. It is easy  to see that $F$ is a generalized Berwald metric while it is not a Berwald metric.

\smallskip

We must mention that Theorem \ref{MainTHM} does not hold  for Finsler metrics of $dim(M)\geq 3$. Denote generic tangent vectors on $\mathbb{S}^3$ as $u {\partial}/{\partial x}+v {\partial}/{\partial y}+  w {\partial}/{\partial z}$. The Finsler functions for Bao-Shen's Randers metrics $F=\alpha+\beta$ are given by the following
\[
F=  \frac{ \, \sqrt{ \mathcal{K} ( c u - z v + y w )^2 + ( z u + c v - x w )^2+ (x v + c w-y u)^2 } \, }{ 1 + x^2 + y^2 + z^2 }
\pm \frac{ \,  \sqrt{ \, \mathcal{K}-1 \, } \ ( c u - z v + y w ) \, } { \, 1 + x^2 + y^2 + z^2 \, } ,
\]
where $\mathcal{K}>1$ is a real constant. For these metrics, we have
\[
b:=\|\beta\|_{\alpha}=\sqrt{1-\frac{1}{\mathcal{K}}}
\]
One can see that, the  one form $\beta$ is not closed, and then $F$ can not be  a Douglas metric.  This family of Randers metrics is generalized Berwald metrics with ${\bf S}=0$  which are not Berwaldian.

\bigskip

In \cite{BM}, it is proved that  every 3-dimensional closed manifold admits a non-Riemannian generalized Berwald metric. However, the closeness condition is very restrictive.  Indeed, for every manifold $M$ with  $dim(M)\geq 3$, one can construct  generalized Berwald  $(\alpha, \beta)$-metrics  that  are not  Berwaldian.
\begin{ex}
 The projective spherical metric on $\mathbb{R}^3$ is given by the following
\be
\alpha:=\frac{\sqrt{(1+||X||^2)||Y||^2-\langle X,Y\rangle^2}}{1+\langle X,X\rangle},\,\, X\in \mathbb{R}^3,\,\,Y\in T_X\mathbb{R}^3,\label{R}
\ee
where $\langle, \rangle$ and $||.||$ denote the Euclidean inner product and norm on $\mathbb{R}^3$, respectively. Put $X=(x,y,z)$ and $Y=(u,v,w)$. Suppose that $\beta=\kappa(b_1u+b_2v+b_3w)$ is a Killing 1-form
of $\alpha$, where $0<\kappa<1$.  By a simple calculation, we get
\begin{eqnarray}
b_1=\frac{A^1_2y+A^1_3z+C^1}{1+\langle X, X \rangle}, \ \  b_2=\frac{A^2_1x+A^2_3z+C^2}{1+\langle X, X \rangle},  \ \
b_3=\frac{A^3_1x+A^3_2y+C^3}{1+\langle X, X \rangle},\label{F1}\
\end{eqnarray}
where $A=(A^i_j)$ is an antisymmetric real matrix,  and $C=(C^i)$ is a constant vector in $\mathbb{R}^3$. Let us put
\be
C=(0,1,0), \ \ \ A^1_2=A^2_3=0, \ \ A^1_3=1.\label{F2}
\ee
In this case, $\beta$ is  a Killing 1-form  with  $\|\beta\|_{\alpha}=\kappa<1$ such that is not  closed. According to  Shen's theorem in \cite{ShC}, a regular $(\alpha, \beta)$-metric is a Berwald metric if and only if $\beta$ is parallel with respect to $\alpha$. Using  the Riemannian metric \eqref{R} and 1-form $\beta$ satisfying \eqref{F1} and  \eqref{F2}, one can construct generalized Berwald $(\alpha, \beta)$-metrics which are not Berwaldian.
\end{ex}

\bigskip

Homogeneous Finsler manifolds are those Finsler manifolds $(M, F)$ that the orbit of the natural  action  of $I(M, F)$ on $M$ at any point of $M$ is the whole $M$. For the class of homogeneous generalized Berwald $(\alpha, \beta)$-metrics, we prove the following.
\begin{cor}\label{cor1}
Let $F=\alpha\phi(s)$, $s=\beta/\alpha$,  be  a two-dimensional  homogeneous generalized Berwald  $(\alpha,\beta)$-metric on a manifold $M$. Then $F$ has isotropic  S-curvature  if and only if it is a Riemannian  metric of constant sectional curvature or a locally Minkowskian metric.
\end{cor}

\bigskip

Let $G$ be a connected Lie group with a bi-invariant Finsler metric ${\bar F}$, and $H$ a closed
subgroup of $G$. Denote the Lie algebras of $G$ and $H$ as $\mathfrak{g}$ and $\mathfrak{h}$, respectively. Let $\rho$ be
the natural projection from $G$ to $M = G/H$. Then  there exists a uniquely defined $G$-invariant metric $F$ on $M$ such that for any $g\in G$, the tangent map $\rho_{*}: \big(T_gG, {\bar F}(g, .)\big) \rightarrow \big(T_{\pi(g)}M, F(\pi(g), .)\big)$ is a submersion (see Lemma 3.1 in \cite{XDn}). The $G$-invariant Finsler metric $F$ is called the normal
homogeneous metric induced by ${\bar F}$. The pair $(M, F)$ is said the normal homogeneous space induced by $\rho : G \rightarrow M = G/H$ and ${\bar F}$.  For normal homogeneous generalized Berwald metrics, we have the following.
\begin{cor}\label{cor2}
Every two-dimensional  normal homogeneous generalized Berwald  $(\alpha,\beta)$-metric is a  Riemannian metric of non-negative constant sectional  curvature or a locally Minkowskian metric.
\end{cor}

A Finsler metric $F$ on a manifold $M$ is called  a generalized normal homogeneous metric if for every $x, y \in M$, there exists a $\delta(x)$-translation  of $(M, F)$ sending $x$ to $y$ (see \cite{ZD}).  In this paper, we consider two-dimensional  homogeneous generalized Berwald Randers metric of generalized normal-type and prove the following.
\begin{cor}\label{cor3}
Every two-dimensional generalized normal homogeneous generalized Berwald Randers metric must be a  Riemannian metric of non-negative constant sectional  curvature or a locally Minkowskian metric.
\end{cor}

\smallskip

Let $(G, \cdot)$ be  a connected Lie group with identity element $e$, and $\lambda_{ g}$ denotes the left  translation by $g\in G$. A Finsler metric $F$ on $G$ is called a left invariant Finsler metric if it satisfies $F \circ (\lambda_{\bf g})_{*} = F$ for any $g\in G$. In \cite{A15}, Aradi proved that left invariant Finsler metrics are generalized Berwald metrics.  For this class of Finsler metrics, we have the following.
\begin{thm}\label{MainTHM2}
Every   left invariant Finsler surface has vanishing S-curvature  if and only if it is  a Riemannian metric of constant sectional curvature.
\end{thm}
By considering Theorems \ref{MainTHM} and \ref{MainTHM2}, it seems that two-dimensional generalized Berwald $(\alpha, \beta)$-metrics with vanishing S-curvature may be only Riemannian. But we do not find a short proof for this conjecture.

\smallskip
A Finsler metric $F$ on a manifold $M$ is called an isotropic Berwald metric, if its Berwald curvature is given by
\begin{eqnarray}\label{IBCurve}
\nonumber{\bf B}_y(u,v,w)=\!\!\!&&\!\!\! cF^{-1}\Big\{{\bf h}(u,v)\big (w-\textbf{g}_y(w,\ell)\ell\big)+{\bf h}(v, w)\big(u-\textbf{g}_y(u,\ell)\ell\big)
\\ \!\!\!&&\!\!\! \qquad\qquad \ \ \ \ \ \ \ \ \qquad +{\bf h}(w, u)\big(v-\textbf{g}_y(v,\ell)\ell\big)+2F{\bf C}_y(u, v, w) \ell\Big\}.
\end{eqnarray}
where ${\bf h}_y(u,v)={\bf g}_y(u,v)-F^{-2}(y){\bf g}_y(y,u){\bf g}_y(y,v)$ is the angular form in direction $y$, ${\bf C}$ denotes the Cartan torsion of $F$ and  $c\in C^{\infty}(M)$. Every Berwald metric is an isotropic Berwald metric with $c=0$. The Funk metrics are isotropic Berwald metrics with $c=1/2$. As a straightforward conclusion of Theorem \ref{MainTHM2},  one can get  the following.

\begin{cor}\label{cor4}
Every   left invariant  isotropic Berwald surface  must be a Riemannian surface of constant sectional curvature.
\end{cor}
%---------------------------------------------------------------------------------------------------------------------------------------
\section{Preliminaries}\label{sectionP}
%---------------------------------------------------------------------------------------------------------------------------------------
Let $M$ be an $n$-dimensional $C^{\infty}$ manifold, $TM=\bigcup_{x \in M}T_{x}M$ the tangent space and $TM_0:=TM-\{0\}$ the slit tangent space of $M$. Let $(M, F)$ be a Finsler manifold. The  following quadratic form $\textbf{g}_y:T_xM \times T_xM\rightarrow \mathbb{R}$ is called fundamental tensor
\[
\textbf{g}_{y}(u,v):={1 \over 2}\frac{\partial^2}{\partial s\partial t} \Big[  F^2 (y+su+tv)\Big]_{s=t=0}, \ \
u,v\in T_xM.
\]

Let  $x\in M$ and $F_x:=F|_{T_xM}$.  To measure the non-Euclidean feature of $F_x$, one can define ${\bf C}_y:T_xM\times T_xM\times T_xM\rightarrow \mathbb{R}$ by
\[
{\bf C}_{y}(u,v,w):={1 \over 2} \frac{d}{dt}\left[\textbf{g}_{y+tw}(u,v)
\right]_{t=0}, \ \ u,v,w\in T_xM.
\]
The family ${\bf C}:=\{{\bf C}_y\}_{y\in TM_0}$  is called the Cartan torsion.

\bigskip

For a   Finsler manifold $(M, F)$,  its induced spray on $TM$ is denoted by ${\bf G}={\bf G}(x, y)$   which in a standard coordinate $(x^i,y^i)$ for $TM_0$ is given by
${\bf G}=y^i {{\partial}/ {\partial x^i}}-2G^i(x,y){{\partial}/ {\partial y^i}}$, where
\[
G^i:=\frac{1}{4}g^{il}\Big[\frac{\partial^2F^2}{\partial x^k \partial y^l}y^k-\frac{\partial F^2}{\partial x^l}\Big].
\]
For a Finsler metric $F$ on an $n$-dimensional manifold $M$, the Busemann-Hausdorff volume form $dV_F = \sigma_F(x) dx^1 \cdots
dx^n$ is defined by
\[
\sigma_F(x) := {{\rm Vol} \big(\Bbb B^n(1)\big)
\over {\rm Vol} \Big \{ (y^i)\in R^n \ \big | \ F \big ( y^i \frac{\partial}{\partial x^i}|_x \big ) < 1 \Big \} } .\label{dV}
\]
Let $G^i$ denote the geodesic coefficients of $F$ in the same
local coordinate system. Then for ${\bf y}=y^i{\partial/ \partial x^i}|_x\in T_xM$, the  S-curvature is defined by
\begin{equation}
 {\bf S}({\bf y}) := {\partial G^i\over \partial y^i}(x,y) - y^i {\partial \over \partial x^i}
\Big [ \ln \sigma_F (x)\Big ].\label{S}
\end{equation}

For a  vector $y \in T_xM_{0}$, the Berwald curvature  ${\bf B}_y:T_xM\times T_xM \times T_xM\rightarrow T_xM$ is defined by
${\bf B}_y(u, v, w):=B^i_{\ jkl}(y)u^jv^kw^l{{\partial }/ {\partial x^i}}|_x$, where
\[
B^i_{\ jkl}:={{\partial^3 G^i} \over {\partial y^j \partial y^k \partial y^l}}.
\]
 $F$ is called a Berwald metric if $\bf{B}=0$. Every Berwald metric satisfies ${\bf S}=0$ (see \cite{TR}).

 For $y\in T_xM$, define the Landsberg curvature ${\bf L}_y:T_xM\times T_xM \times T_xM\rightarrow \mathbb{R}$  by
\[
{\bf L}_y(u, v,w):=-\frac{1}{2}{\bf g}_y\big({\bf B}_y(u, v, w), y\big).
\]
A Finsler metric $F$ is called a Landsberg  metric if ${\bf L}=0$.

Taking a trace of Berwald curvature ${\bf B}$ give us  the mean of Berwald curvature ${\bf E}$ which is defined   by  ${\bf E}_y:T_xM\times T_xM \rightarrow \mathbb{R}$, where
\be
{\bf E}_y (u, v) := {1\over 2} \sum_{i=1}^n g^{ij}(y) {\bf g}_y \big ({\bf B}_y (u, v, \partial_i ) , \partial_j \big).
\ee
where $\{\partial_i\}$ is a basis for $T_xM$ at $x\in M$. In local coordinates,   ${\bf E}_y(u, v):=E_{ij}(y)u^iv^j$, where
\[
E_{ij}:=\frac{1}{2}B^m_{\ mij}.
\]
\newpage

Taking a horizontal derivation of the mean of Berwald curvature ${\bf E}$ along Finslerian geodesics give us the $H$-curvature ${\bf H}={\bf H}(x, y)$ which is defined by  ${\bf H}_y=H_{ij}dx^i\otimes dx^j$, where
\[
H_{ij}:= E_{ij|m}y^m.
\]
Here,  $``|"$ denotes the horizontal covariant differentiation with respect to the Berwald connection of $F$.

For a non-zero vector $y \in T_xM_0$, one can define   ${\bf D}_y:T_xM\times T_xM \times T_xM\rightarrow T_xM$ by  ${\bf D}_y(u,v,w):=D^i_{\ jkl}(y)u^iv^jw^k{\partial}/{\partial x^i}|_{x}$, where
\be
D^i_{\ jkl}:=\frac{\partial^3}{\partial y^j\partial y^k\partial y^l}\Bigg[G^i-\frac{2}{n+1}\frac{\partial G^m}{\partial y^m} y^i\Bigg].\label{Douglas1}
\ee
$\bf D$ is called the Douglas curvature.  $F$ is called a Douglas metric if $\bf{D}=0$.

\bigskip
For a non-zero vector $y \in T_xM_{0}$, the Riemann curvature is a family of linear
transformation $\textbf{R}_y: T_xM \rightarrow T_xM$  which is defined by
$\textbf{R}_y(u):=R^i_{k}(y)u^k {\partial / {\partial x^i}}$, where
\be
R^i_{k}(y)=2{\partial G^i \over {\partial x^k}}-{\partial^2 G^i \over
{{\partial x^j}{\partial y^k}}}y^j+2G^j{\partial^2 G^i \over
{{\partial y^j}{\partial y^k}}}-{\partial G^i \over {\partial
y^j}}{\partial G^j \over {\partial y^k}}.\label{Riemannx}
\ee
The family $\textbf{R}:=\{\textbf{R}_y\}_{y\in TM_0}$ is called the Riemann curvature.

For a flag $P:={\rm span}\{y, u\} \subset T_xM$ with the flagpole $y$, the flag curvature ${\bf
K}={\bf K}(P, y)$ is defined by
\be
{\bf K}(x, y, P):= {{\bf g}_y \big(u, {\bf R}_y(u)\big) \over {\bf g}_y(y, y) {\bf g}_y(u,u)
-{\bf g}_y(y, u)^2 }.\label{FC0}
\ee
The flag curvature ${\bf K}(x, y, P)$ is a function of tangent planes $P={\rm span}\{ y, v\}\subset T_xM$.   A Finsler metric $F$ is of scalar flag curvature if ${\bf K}(x, y, P)={\bf K}(x, y)$ is independent of $ P$. Also, $F$ is called of isotropic  and  constant flag curvature if ${\bf K}={\bf K}(x)$ and  ${\bf K}=constant$, respectively.

\bigskip

Throughout  this paper, we use  the Berwald connection on Finsler manifolds. The pullback bundle $\pi^*TM$ admits a unique linear connection, called the Berwald connection. Let $(M, F)$ be an $n$-dimensional Finsler manifold. Let $\{e_j\}$ be a local
frame for $\pi^*TM$, $\{\omega^i, \omega^{n+i}\}$ be the corresponding local coframe for
$T*(TM_0)$ and $\{\omega^i_j\}$ be the set of local Berwald  connection forms with respect
to $\{e_j\}$. Then the connection forms are characterized by the structure equations as follows
\begin{itemize}
  \item Torsion freeness:
  \begin{eqnarray}
d\omega^i = \omega^j \wedge \omega^i_{\ j}.
\end{eqnarray}
  \item Almost metric compatibility:
  \begin{eqnarray}
dg_{ij}- g_{kj}\omega^k_{\ i} - g_{ik}\omega^k_{\ j} = -2L_{ijk}\omega^k + 2C_{ijk}\omega^{n+k},
\end{eqnarray}
where $\omega^i := dx^i$ and $\omega^{n+k}:= dy^k + y^j\omega^k_{\ j}$.
\end{itemize}
The horizontal and vertical covariant derivations with respect to the Berwald connection respectively are  denoted by ``$|$'' and  ``, ". For more details, one can see \cite{ShDiff}.

%-------------------------------------------------------------------------------------------------------------------------------------------------------------------------
\section{Proof of  Theorems \ref{MainTHM} and \ref{MainTHM2}}
%-------------------------------------------------------------------------------------------------------------------------------------------------------------------------
For an $(\alpha, \beta)$-metric, let us define $b_{i;j}$ by $b_{i;j}\theta^j:=db_i-b_j\theta^j_i$, where
$\theta^i:=dx^i$ and $\theta^j_i:=\gamma^j_{ik}dx^k$ denote
the Levi-Civita connection form of $\alpha$. Let
\begin{eqnarray}
r_{ij}:=\frac{1}{2}(b_{i;j}+b_{j;i}), \ \ \ s_{ij}:=\frac{1}{2}(b_{i;j}-b_{j;i}),\ \ \ r_{i0}: = r_{ij}y^j, \  \ r_{00}:=r_{ij}y^iy^j,\ \ r_j := b^i r_{ij},\label{rs1} \\
s_{i0}:= s_{ij}y^j, \  \ \ s_j:=b^i s_{ij}, \ \ \ s^i_{\ j}=a^{im}s_{mj}, \ \ \ s^i_{\ 0}=s^i_{\ j}y^j,\ \ \  r_0:= r_j y^j,\ \  \ \  s_0 := s_j y^j.\label{rs2}
\end{eqnarray}
Put
\begin{eqnarray}
 Q\!\!\!\!&:=&\!\!\!\! \frac{\phi'}{\phi- s\phi},\ \ \ \  \ \Delta:=1+sQ+(b^2-s^2)Q',\ \ \ \ \ \Theta:=\frac{Q-sQ'}{2\Delta},\label{del}\\
\Phi\!\!\!\!&:=&\!\!\!\!  - (Q -s Q')(n\Delta + 1 + sQ) -(b^2-s^2) (1+sQ) Q'',\label{phi}\\
\nonumber\Psi\!\!\!\!&:=&\!\!\!\! \frac{\phi''}{2\big[(\phi -s\phi')+(b^2-s^2)\phi''\big]},
\end{eqnarray}
where  $b:=\|\beta\|_\alpha$. Let $G^i=G^i(x,y)$ and $\bar{G}^i_{\alpha}=\bar{G}^i_{\alpha}(x,y)$ denote the
coefficients  of $F$ and  $\alpha$, respectively, in the same coordinate system. By definition, we have
\begin{eqnarray}
G^i=G^i_{\alpha}+\alpha Q s^i_0+\alpha^{-1}(r_{00}-2Q\alpha s_0)(\Theta y^i+\alpha\Psi b^i).\label{G1}
\end{eqnarray}
Clearly, if $\beta $ is parallel with respect to $\alpha$, that is $r_{ij}=s_{ij}=0$, then $G^i ={G}^i_{\alpha}=\gamma^i_{jk}(x)y^jy^k$ are quadratic in $y$. In this case, $F$ reduces to a Berwald metric.
\bigskip

For an $(\alpha, \beta)$-metric $F=\alpha\phi(s)$, $s=\beta/\alpha$,  on an $n$-dimensional manifold $M$, the $S$-curvature is given by
\be
{\bf S}=\Big[2\Psi-\frac{f'(b)}{bf(b)}\Big](r_0+s_0)-\frac{\Phi}{2\alpha\Delta^2}(r_{00}-2\alpha Qs_0),\label{SC}
\ee
where
\begin{eqnarray*}
f(b):=\frac{\int^\pi_0\sin^{n-2}t\ T(b\cos t)dt}{\int^\pi_0\sin^{n-2}tdt},\ \  \ \ T(s):=\phi(\phi-s\phi')^{n-2}\big[(\phi-s\phi')+(b^2-s^2)\phi'' \big].
\end{eqnarray*}
For more details, see \cite{ChSh3}.

\bigskip
To prove Theorem \ref{MainTHM}, we need the following key lemma.
\begin{lem}{\rm (\cite{BM}\cite{TB})}\label{M}
Let $F= \alpha\phi(s)$, $s = \beta/\alpha$, be an  $(\alpha, \beta)$-metric on a manifold $M$. Then  $F$ is a generalized Berwald metric if and only if $\beta$ has constant length with respect to $\alpha$.
\end{lem}
\begin{proof}
In \cite{TB}, Tayebi-Barzegari showed that every $(\alpha, \beta)$-metric  $F=\alpha\phi(s)$, $s=\beta/\alpha$, with sign property is a generalized Berwald metric if and only if the dual vector field $\beta^\sharp$ is of constant Riemannian length.  In \cite{I}, Ivanov  proved that a two-dimensional Finsler metric is a generalized Berwald metric if and only if it is monochromatic, i.e., Finsler metrics such that each two tangent spaces are isomorphic as normed spaces.     In \cite{BM}, Bartelme${\ss}$-Matveev extended his result for $n$-dimensional Finsler metric. It follows that an $(\alpha, \beta)$-metric  is a generalized Berwald metric if and only if dual vector field $\beta^\sharp$ is of constant Riemannian length.
\end{proof}

\bigskip

In \cite{ChSh3}, Cheng-Shen  characterized $(\alpha,\beta)$-metrics with isotropic
S-curvature on a manifold $M$ of dimension $n\geq 3$. Soon, they found that their result holds for the class of  $(\alpha,\beta)$-metrics with  constant length one-forms, only. Here, we give a characterization of  the class of generalized Berwald  metrics with vanishing S-curvature.
\begin{lem}\label{LS}
Let $F= \alpha\phi(s)$, $s = \beta/\alpha$, be a  generalized Berwald  $(\alpha, \beta)$-metric on an $n$-dimensional manifold $M$. Then ${\bf S}=0$ if and only if $F$ is Riemannian or  $\beta$ is a Killing form with  constant length.
\end{lem}
\begin{proof}
Let $b:=\|\beta\|_\alpha=\sqrt{a^{mj}b_jb_m}=\sqrt{b_mb^m}$. Then, the following holds
\be
\frac{\partial b}{\partial x^i} =\frac{1}{b} b^mb_{m|i}= \frac{1}{b}(r_i + s_i).\label{bc}
\ee
By Lemma \ref{M}, we have $b = constant$. Then, by \eqref{bc} we obtain $r_i+s_i=0$. In this case,  \eqref{SC} reduces to following
\be
{\bf S}=-\frac{\Phi}{2\alpha\Delta^2}(r_{00}-2\alpha Qs_0),\label{SS}
\ee
By \eqref{SS}, ${\bf S}=0$ if and only if  $\Phi=0$ or $\beta$ satisfies
\begin{equation} \label{condition 1l}
r_{00}=2\alpha Q s_0.
\end{equation}
In the case of $\Phi=0$, $F$ reduces to a Riemannian metric (see Proposition 2.2 in \cite{CWW}).  Now, let \eqref{condition 1l} holds.  We are going to simplify (\ref{condition 1l}). For this aim,  one can   change the $y$-coordinates $(y^i)$, $i=1, \cdots, n$, at a point to  the polar coordinates $ (s,u^A)$, where  $A=2, \cdots, n$ (see \cite{ChSh3}). For an arbitrary and fix point $x\in M$, let us take an orthonormal basis $e_i$ at $x$ such that the Riemannian metric is written as $\alpha=\sqrt{\sum_{i=1}^n (y^i)^2}$ and  its related one-form is given by $\beta=by^1$, where $b:=||\beta||_\alpha$. Let us fix an arbitrary number $s$ such that $|s| < b$. Define
\[
{\bar \alpha}=\sqrt{\sum_{A=2}^n (y^A)^2}.
\]
Then, by $\beta = s\alpha$ we get
\begin{equation}
y^1=\frac{s}{\sqrt{b^2-s^2}} {\bar \alpha}, \ \ \  \ y^A=u^A.\label{44}
\end{equation}
Also, we have
\begin{equation}
\alpha =\frac{b}{\sqrt{b^2-s^2} }\bar{\alpha}, \ \ \ \ \
\beta = \frac{bs}{\sqrt{b^2-s^2} } \bar{\alpha}. \label{45}
\end{equation}
Let us put
\begin{eqnarray*}
\bar{r}_{10}:= \sum_{A=2}^nr_{1A} y^{A}, \ \  \bar{s}_{10}:= \sum_{A=2}^ns_{1A} y^{A},\ \ \bar{r}_{00} := \sum_{A,B=2}^n r_{AB}y^{A} y^{B},\ \
\bar{r}_0 := \sum_{A=2}^n r_{A} y^{A}, \ \ \bar{s}_{0}:=\sum_{A=2}^n s_{A} y^{A}.
\end{eqnarray*}
Then we get the following useful relations
\begin{eqnarray}
&& r_1 = b r_{11}, \ \ \ \ r_{A} = b r_{1A}, \ \ \ s_1 = 0, \ \ \ \ s_{A} = b s_{1A},\label{47}\\
&&r_{00}=\frac{s^{2}}{b^{2}-s^{2}}\bar{\alpha}^{2}r_{11}+\frac{2s}{\sqrt{b^{2}-s^{2}}}\bar{\alpha}\bar{r}_{10}+\bar{r}_{00},\label{48}\\
&&r_{10}=\frac{s}{\sqrt{b^{2}-s^{2}}}\bar{\alpha}r_{11}+\bar{r}_{10}, \ \ \ \ \ \ \ s_{0}=\bar{s}_{0}=b\bar{s}_{10}. \label{49}
\end{eqnarray}
Using (\ref{45})-(\ref{49}),  the equation (\ref{condition 1l})  can be written as follows
\begin{eqnarray}
\bar{r}_{00}+\frac{s^{2}}{b^{2}-s^{2}}\bar{\alpha}^{2}\ r_{11}=\frac{2}{\sqrt{b^2-s^2} }\left[b^2 Q \bar{s}_{10}-s\bar{r}_{10}\right]\bar{\alpha}.\label{50}
\end{eqnarray}
By (\ref{50}), we get two following relations
\begin{eqnarray}
&& \bar{r}_{00}+  \frac{s^{2}}{b^{2}-s^{2}}\bar{\alpha}^{2}r_{11}=0, \label{51}\\
&& b^2Q \bar{s}_{10}-s\bar{r}_{10}=0.\label{52}
\end{eqnarray}
On the other hand, (\ref{45}) implies that
\begin{eqnarray}
\frac{s^{2}}{b^{2}-s^{2}}\bar{\alpha}^{2}-\frac{1}{b^2}\beta^2=0. \label{53}
\end{eqnarray}
By (\ref{51}) and (\ref{53}), we get
\begin{eqnarray}
b^2\bar{r}_{00}+\beta^2r_{11}=0. \label{54}
\end{eqnarray}
The following hold
\[
\frac{\partial \bar{r}_{00}}{\partial y^1}=0, \ \ \ \ \ \  \frac{\partial \beta}{\partial y^1}=b.
\]
Then by differentiating (\ref{54}) with respect to $y^1$ we have
${\beta}/{b}r_{11}=0$. Thus $r_{11}=0$ and by putting it in (\ref{54}), we get $\bar{r}_{00}=0$. Putting these relations  in (\ref{48}) and (\ref{49})  imply that
\begin{eqnarray}
&&r_{00}=\frac{2s\bar{\alpha}}{\sqrt{b^{2}-s^{2}}}\bar{r}_{10}=\frac{2\beta}{b} \bar{r}_{10},\label{58}\\
&&r_{10}=\bar{r}_{10}. \label{59}
\end{eqnarray}
Now, by considering \eqref{58} and \eqref{59},  we divide the problem into two cases: (a) $\bar{r}_{10} = 0$ and (b) $\bar{r}_{10}\neq 0$.\\\\
{\bf Case (a):  $\bar{r}_{10} = 0$}. In this case, by  (\ref{58})  we get $r_{ij}=0$. Putting it  in  (\ref{condition 1l})   implies that $s_i=0$. In this case, $\beta$ reduces to a Killing one-form of constant length with respect to $\alpha$.\\\\
{\bf Case (b):  $\bar{r}_{10} \neq 0$}. We have ${\partial \bar{r}_{10}}/{\partial y^1}=0$ and ${\partial \bar{s}_{10}}/{\partial y^1}=0$.  Thus, differentiating (\ref{52}) with respect to $y^1$ yields
\begin{eqnarray}
(s)_{y^1}\bar{r}_{10}=b^2 (Q)_{y^1} \bar{s}_{10}.\label{63}
\end{eqnarray}
Contracting (\ref{52}) with $(s)_{y^1}$  give us
\begin{eqnarray}
s(s)_{y^1}\bar{r}_{10}= b^2Q (s)_{y^1}\bar{s}_{10}.\label{64}
\end{eqnarray}
By (\ref{63})  and (\ref{64}), we get
\begin{eqnarray}
\Big[(s)_{y^1}  Q-s(Q)_{y^1}\Big]\bar{s}_{10}=0.\label{65}
\end{eqnarray}
According to  (\ref{65}), we get $\bar{s}_{10}=0$ or $Q(s)_{y^1}=s(Q)_{y^1}$. Let $\bar{s}_{10}=0$ holds. Then (\ref{52}) reduces to $s=0$, which is impossible.  Then, we have
\[
Q(s)_{y^1}=s(Q)_{y^1},
\]
which is equal to
\begin{eqnarray}
\frac{(Q)_{y^1}}{Q}=\frac{(s)_{y^1}}{s}.\label{70}
\end{eqnarray}
Using $s_{y^1}\neq 0$ and $(Q)_{y^1}=s_{y^1}(Q)_s$,  then (\ref{70}) give us
\[
\frac{(Q)_s}{Q}=\frac{1}{s}
\]
which yields
\[
\ln(Q)-\ln(s)=c,
\]
where $c$ is a real constant. Thus $Q=ks$, where $k$ is a non-zero real constant. In this case, we get $\phi=\sqrt{1+ks^2}$ which shows that $F$ is a Riemannian metric. This completes the proof.
\end{proof}

\bigskip
Here, we solve an ODE which will appear in the proof of Theorem \ref{MainTHM}.
\begin{lem}\label{alaki}
Let $F= \alpha\phi(s)$, $s = \beta/\alpha$, be an  $(\alpha, \beta)$-metric on a manifold $M$. Suppose that $\phi$ satisfies following
\begin{eqnarray}
\alpha \Theta_2+2\Lambda_1\Theta_1+\Lambda_2Q=0,\label{eq}
 \end{eqnarray}
where
\begin{eqnarray*}
\Lambda_1: = b^i\alpha_{y^i}, \ \ \ \Lambda_2:=b^ib^j\alpha_{y^iy^j}, \ \ \ \ \Theta_1: = b^iQ_{y^i}, \  \ \Theta_2:=b^ib^jQ_{y^iy^j}.
\end{eqnarray*}
Then $F$ is a singular Finsler metric given by
\begin{eqnarray}
\phi=c\exp\Bigg[\int_0^s{\frac{k_1t+k_2\sqrt{b^2-t^2}}{1+t\big(k_1t+k_2\sqrt{b^2-t^2}\big)}dt}\Bigg],\label{uni}
\end{eqnarray}
where $k_1$ and $k_2$ are real constants, and $c>0$ is a non-zero constant.
\end{lem}
\begin{proof}
Let us put
\[
A_{jk}:=\alpha^2a_{jk}-y_jy_k.
\]
Then, the followings hold
\begin{eqnarray*}
&&\alpha_{y^i}=\frac{1}{\alpha}y_i,\\
&&\alpha_{y^jy^k}=\frac{1}{\alpha^3}A_{jk},\\
&&\alpha_{y^jy^ky^l}=-\frac{1}{\alpha^5}\Big[A_{jk}y_l+A_{jl}y_k+A_{lk}y_j\Big].
\end{eqnarray*}
Also, one can obtain the following
\begin{eqnarray}
&& \Lambda_1=s,\label{a1x}\\
&& \Lambda_2=\frac{1}{\alpha}(b^2-s^2),\label{a2x}\\
&&\Theta_1= \frac{1}{\alpha}(b^2-s^2)Q',\label{a3x}\\
&&\Theta_2=\frac{1}{\alpha^2}(b^2-s^2)\big[(b^2-s^2)Q''-3sQ'\big].\label{005x}
\end{eqnarray}
Suppose that  \eqref{eq} holds. Putting  (\ref{a1x})-(\ref{005x}) into (\ref{eq}) yield
\begin{eqnarray*}
Q''-\frac{s}{b^2-s^2} Q'+\frac{1}{b^2-s^2}Q=0,
\end{eqnarray*}
which implies that
\be
Q=k_1s+k_2\sqrt{b^2-s^2},\label{Q}
\ee
where $k_1$ and $k_2$ are real constants. Considering \eqref{del}, the equation \eqref{Q} is equal to following
\begin{eqnarray}
\Big[ 1+k_1s^2+k_2s\sqrt{b^2-s^2} \Big]\phi'=\Big[k_1s+k_2\sqrt{b^2-s^2}\Big]   \phi.\label{alaki4}
\end{eqnarray}
By \eqref{alaki4}, we get  \eqref{uni}. It is an almost regular $(\alpha, \beta)$-metric, namely, it is singular in two directions ${\bf y}=(\pm b, 0, 0)\in T_xM$ at any point $x$ (for more details, see \cite{ShC}).
\end{proof}

\bigskip

The function $\phi$ in \eqref{uni} is specifically has been seen for the first time in Asanov's paper \cite{As} as follows
\begin{eqnarray}
\phi=c\exp\Bigg[\int_0^s{\frac{k\sqrt{b^2-t^2}}{1+tk\sqrt{b^2-t^2}}dt}\Bigg].\label{unias}
\end{eqnarray}
Then, its more general form \eqref{uni} found by Shen in \cite{ShC}, where he looked for unicorn metrics, namely the Landsberg metrics which are not Berwaldian.  Shen  realized  that the function $\phi=\phi(b, s)$ in \eqref{uni} can be expressed in terms of elementary functions. See (7.2) in \cite{ShC}.

\bigskip

Here, we show that the Douglas curvature of Finsler surfaces satisfies a special relation. More precisely, we prove the following.
\begin{lem} \label{LD}
The Douglas curvature of any Finsler surface $(M, F)$ satisfies
\be
D^i_{\ jkl|s}y^s = \mathfrak{D}_{jkl}y^i\label{DD}
\ee
for some tensor $\mathfrak{D}=\mathfrak{D}_{ijk}dx^i\otimes dx^j\otimes dx^k$ which are homogeneous of degree -1 in $y$.
\end{lem}
\begin{proof}
By definition, the Douglas curvature of a two-dimensional Finsler metric is given by
\begin{equation}
D^i_{\ jkl}=B^i_{\ jkl}-\frac{2}{3}\Big\{E_{jk}\delta^i_{\ l}+E_{kl}\delta^i_{\ j}+E_{lj}\delta^i_{\ k}+E_{jk,l}y^i\Big\}.\label{GDW2}
\end{equation}
Taking a horizontal derivation of (\ref{GDW2}) along Finslerian geodesics  and using $y^i_{\ |s}=0$ give us the following
\begin{equation}
D^i_{\ jkl|m}y^m=B^i_{\ jkl|m}y^m-\frac{2}{3}\Big\{H_{jk}\delta^i_{\ l}+H_{kl}\delta^i_{\ j}+H_{lj}\delta^i_{\ k}+E_{jk,l|m}y^my^i\Big\}.\label{DW1}
\end{equation}
Every Finsler surface is of scalar flag curvature ${\bf K}={\bf K}(x, y)$. Then, by (11.24) in \cite{ShDiff} we have
\begin{eqnarray}
\nonumber B^i_{\ jml|k}y^k =2\textbf{K}C_{jlm}y^i- \frac{1}{3}\Big\{y_l\delta^i_m+y_m\delta^i_l-2g_{lm}y^i \Big \}\textbf{K}_{j} - \frac{1}{3}\Big\{y_j\delta^i_m+y_m\delta^i_j-2g_{jm}y^i \Big \}\textbf{K}_{l}\nonumber\\
- \frac{1}{3}\Big\{y_j\delta^i_l+y_l\delta^i_j-2g_{jl}y^i \Big \}\textbf{K}_{m} - \frac{1}{3}F^2\Big\{\textbf{K}_{jm}h^i_l+\textbf{K}_{jl}h^i_m+\textbf{K}_{lm}h^i_j\Big\},\label{HW21}
\end{eqnarray}
where $y_i=FF_{y^i}$, $\textbf{K}_{j}:=\textbf{K}_{y^j}$ and $\textbf{K}_{jk}:=\textbf{K}_{y^jy^k}$. Taking a trace of  (\ref{HW21}) yields
\be
H_{jl}=-\frac{1}{2}\Big\{y_l\textbf{K}_{j}+y_j\textbf{K}_{l}+F^2\textbf{K}_{jl}\Big\}.\label{HW22}
\ee
By putting (\ref{HW21}) and  (\ref{HW22}) in  (\ref{DW1}) we get
\begin{eqnarray}
\nonumber D^i_{\ jkl|m}y^m=\!\!\!&&\!\!\! \frac{1}{3}\Big\{6\textbf{K}C_{jkl} +2\big(g_{kl}\textbf{K}_{j}+g_{kj}\textbf{K}_{l}+g_{jl}\textbf{K}_{k}\big)+\big( y_{j}\textbf{K}_{kl}+  y_{k}\textbf{K}_{jl}+ y_{l}\textbf{K}_{kj}\big)\\
\!\!\!&&\!\!\!\qquad\qquad\qquad\qquad \qquad\qquad\qquad\qquad \qquad\qquad\qquad\quad  - 2E_{jk,l|m}y^m\Big\}y^i.\label{dh}
\end{eqnarray}
\eqref{dh} give us \eqref{DD}.
\end{proof}

\bigskip
Now, we show  that the covariant derivative of Berwald  curvature of 2-dimensional  generalized Berwald $(\alpha,\beta)$-metric with vanishing S-curvature satisfies an interesting relation which will play an important role in the proof of  Theorem \ref{MainTHM}.

\begin{lem}\label{B2}
The Berwald curvature of any non-Riemannian 2-dimensional  generalized Berwald $(\alpha,\beta)$-metric with vanishing S-curvature  satisfies following
\begin{eqnarray}
\nonumber h^m_pB^p_{\ jkl|s}y^s\!\!\!&=&\!\!\ s^m_{\ 0|0}(\alpha_{jkl}Q+\alpha_{jk}Q_l+\alpha_{lk}Q_j+\alpha_{lj}Q_k+\alpha Q_{jkl}+\alpha_lQ_{jk}+\alpha_jQ_{lk}+\alpha_kQ_{jl})\\
\nonumber\!\!\!&&\!\!\ +s^m_{\ 0}(\alpha_{jkl}Q_{|0}+\alpha_{jk}Q_{l|0}+\alpha_{lk}Q_{j|0}+\alpha_{lj}Q_{k|0}+\alpha Q_{jkl|0}+\alpha_lQ_{jk|0}+\alpha_jQ_{lk|0}\\
\!\!\!&&\!\!\ +\alpha_kQ_{jl|0})+A^m_{\ l}X_{jk} +A^m_{\ j}X_{lk}+A^m_{\ k}X_{jl}+B^m_{\ l}Y_{jk} +B^m_{\ j}Y_{lk} +B^m_{\ k}Y_{jl}=0,\label{048} \ \ \ \ \ \ \
\end{eqnarray}
where
\begin{eqnarray*}
&&X_{jk}:=Q_{|0}\alpha_{jk}+Q_{k|0}\alpha_j+Q_{j|0}\alpha_k+ Q_{jk|0}\alpha,\\
&&Y_{jk}:=Q\alpha_{jk}+Q_{k}\alpha_j+Q_{j}\alpha_k+\alpha Q_{jk},\\
&&A^m_{\ l}:=s^m_{\ l}-F^{-2}s^0_{\ l}y^m,\\
&&B^m_{\ l}:=A^m_{\ l|s}y^s=s^m_{\ l|0}-F^{-2}s^0_{\ l|0}y^m.
\end{eqnarray*}
\end{lem}
\begin{proof}
By Lemma \ref{LS}, we have $r_{ij}=0$ and $s_j=0$. In this case, (\ref{G1}) reduces to following
\begin{eqnarray}
G^i=G^i_{\alpha}+\alpha Q s^i_{\ 0}.\label{045}
\end{eqnarray}
Taking three vertical derivations of  (\ref{045})  with respect to $y^j$, $y^l$ and $y^k$ give us
\begin{eqnarray}
\nonumber B^i_{\ jkl}=\!\!\!\!&&\!\!\!\! s^i_{\ 0}\big\{\alpha Q_{jkl}+\alpha_lQ_{jk}+\alpha_jQ_{lk}+\alpha_kQ_{jl}+\alpha_{jkl}Q+\alpha_{jk}Q_l+\alpha_{lk}Q_j+\alpha_{lj}Q_k\big\}\\
\nonumber \!\!\!\!&+&\!\!\!\! s^i_{\ j}\big\{Q\alpha_{lk}+Q_{k}\alpha_l+Q_{l}\alpha_k+\alpha Q_{lk}\big\}+ s^i_{\ k}\big\{Q\alpha_{jl}+Q_{j}\alpha_l+Q_{l}\alpha_j+\alpha Q_{jl}\big\}\\
\!\!\!\!&&\!\!\!\! \qquad\qquad\qquad\qquad\qquad\qquad\qquad\ +s^i_{\ l}\big\{Q\alpha_{jk}+Q_{k}\alpha_j+Q_{j}\alpha_k+\alpha Q_{jk}\big\}.\label{046}
\end{eqnarray}
Contracting (\ref{046}) with $h^m_i$ implies that
\begin{eqnarray}
\nonumber h^m_iB^i_{\ jkl}=   \big\{\alpha_{jkl}Q+ \alpha_{jk}Q_l\!\!\!\!&+&\!\!\!\! \alpha_{lk}Q_j+\alpha_{lj}Q_k+\alpha Q_{jkl}+\alpha_lQ_{jk}+\alpha_jQ_{lk}+\alpha_kQ_{jl}\big\}s^m_{\ 0}\\
\nonumber\!\!\!\!\!\!&+&\!\!\!\! \big\{Q\alpha_{jk}+Q_{k}\alpha_j+Q_{j}\alpha_k+\alpha Q_{jk}\big\}(s^m_{\ l}-F^{-2}s^0_{\ l}y^m)\\
\!\!\!\!\!\!&+&\!\!\! \nonumber \big\{Q\alpha_{lk}+Q_{k}\alpha_l+Q_{l}\alpha_k+\alpha Q_{lk}\big\}(s^m_{\ j}-F^{-2}s^0_{\ j}y^m)\\
\!\!\!\!\!\!&+&\!\!\!  \big\{Q\alpha_{jl}+Q_{j}\alpha_l+Q_{l}\alpha_j+\alpha Q_{jl}\big\}(s^m_{\ k}-F^{-2}s^0_{\ k}y^m).\label{047}
\end{eqnarray}
On the other hand, by taking a horizontal derivation of Douglas curvature along Finslerian geodesics and contracting the result with $h^m_i$, we get the following
\begin{equation}
h^m_iD^i_{\ jkl|s}y^s=h^m_iB^i_{\ jkl|s}y^s-\frac{2}{3}\Big\{H_{jk}h^m_l+H_{kl}h^m_j+H_{lj}h^m_k\Big\}.\label{GD3}
\end{equation}
Contracting \eqref{DD} with $h^m_i$ yields
\begin{equation}
h^m_iD^i_{\ jkl|s}y^s=0.\label{GD3x}
\end{equation}
Since ${\bf S}=0$, then by definition  we get ${\bf H}=0$.  Thus, (\ref{GD3}) and (\ref{GD3x}) imply that
\be
h^m_iB^i_{\ jkl|s}y^s=0.\label{GD4}
\ee
We have  $h^m_{\ i|s}=0$. Then,  by considering  (\ref{GD4}), we have
\be
\big(h^m_iB^i_{\ jkl}\big)_{|s}y^s=h^m_iB^i_{\ jkl|s}y^s=0.\label{GD4x}
\ee
Therefore, taking a horizontal derivation of (\ref{047}) along Finslerian geodesic and considering \eqref{GD4x} give us  \eqref{048}.
\end{proof}

\bigskip

\noindent
{\bf Proof of Theorem \ref{MainTHM}:} Taking a horizontal derivation of  $y_ms^m_{\ 0}=0$  with respect to the Berwald connection of $F$ implies that
\be
y_{m|0}s^m_{\ 0}+y_ms^m_{\ 0|0}=0.\label{alaki2}
\ee
 Since $y_{m|0}=0$, then \eqref{alaki2} reduces to following
\be
y_ms^m_{\ 0|0}=0.\label{ys}
\ee
By contracting (\ref{048}) with $y_m$ and considering \eqref{ys},   we get
\begin{eqnarray}
s^0_{\ l}X_{jk}+ s^0_{\ j}X_{lk}+s^0_{\ k}X_{jl}+s^0_{\ l|0}Y_{jk}+s^0_{\ j|0} Y_{lk}+s^0_{\ k|0}Y_{jl}=0.\label{0.49}
\end{eqnarray}
Since $s_j=0$, then one can get
\[
0=(s_j)_{|0}=(r_{m0}+s_{m0})s^m_{\ j}+b_ms^m_{\ j|0}
\]
 which considering  $r_{ij}=0$, it reduces to following
\begin{eqnarray}
b_ms^m_{\ j|0}=-s_{m0}s^m_{\ j}.\label{050}
\end{eqnarray}
Multiplying  (\ref{050}) with $y^j$  yields
\be
b_is^i_{\ 0|0}+s_{i0}s^i_{\ 0}=0.\label{ee1}
\ee
Also, contracting  (\ref{050}) with $b^j$ implies that
\be
b_is^i_{\ j|0}b^j=0.\label{ee2}
\ee
Multiplying (\ref{0.49}) with $b^jb^kb^l$  and  considering (\ref{ee1}) and (\ref{ee2})  give us
\be
(\alpha \Theta_2+2\Lambda_1\Theta_1+\Lambda_2Q)s_{m0}s^{m}_{\ 0}=0.\label{alaki3}
\ee
By  \eqref{alaki3}, we have two cases: if $s^{m}_{\ 0}s_{m0}=0$, since $\alpha$ is a positive-definite metric, then we find that $\beta$ is closed. Therefore,  by  (\ref{046}) we conclude that $F$  reduces to a Berwald metric. By Szabo's rigidity result for Finsler surfaces, $F$ reduces to a locally Minkowskian metric  or a Riemannian metric. On the other hand, every Finsler surface  has scalar flag curvature ${\bf K}={\bf K}(x, y)$. According to Akbar-Zadeh theorem in \cite{AZ}, a Finsler manifold $(M, F)$ of scalar flag curvature ${\bf K}={\bf K}(x, y)$ has isotropic flag curvature ${\bf K}={\bf K}(x)$ if and only if it has vanishing H-curvature ${\bf H}=0$. Thus, the obtained Riemannian metric has isotropic sectional curvature.

Now, suppose that  $F$ is not  a Riemannian metric nor a locally Minkowskian metric. Then, by \eqref{alaki3} we have  $\alpha \Theta_2+2\Lambda_1\Theta_1+\Lambda_2Q=0$. By Lemma \ref{alaki} we obtain  \eqref{unix}. In this case, since ${\bf S}=0$, then  $F$ can not be a Douglas metric. On the other hand, the Berwald curvature of 2-dimensional Finsler manifold is given by
\be
B^i_{\ jkl}=-\frac{2}{F^{2}}L_{jkl}y^i+\frac{2}{3}\Big\{E_{jk}h^i_l+E_{kl}h^i_j+E_{jl}h^i_k\Big\}.\label{2L}
\ee
See the relation (15)  in \cite{TP}. If $F$ is a Landsberg metric then by considering ${\bf S}=0$, \eqref{2L} implies that $F$ is a Berwald metric. This is a contradiction. Then, \eqref{unix} is a generalized Berwald metric  which is not Berwald, Landsberg  nor Douglas metric.
\qed

\bigskip

\bigskip

\noindent
{\bf Proof of Corollary  \ref{cor1}:} In \cite{XDg}, Xu-Deng proved that every homogeneous Finsler metric of isotropic S-curvature has vanishing S-curvature. Then, by assumption we get ${\bf S}=0$.  The Akbar-Zadeh theorem in \cite{AZ} stated that a Finsler manifold $(M, F)$ of scalar flag curvature ${\bf K}={\bf K}(x, y)$ has isotropic flag curvature ${\bf K}={\bf K}(x)$ if and only if it has vanishing H-curvature ${\bf H}=0$. On the other hand, every Finsler surface  has scalar flag curvature ${\bf K}={\bf K}(x, y)$. Thus, by Akbar-Zadeh theorem   we get ${\bf K}={\bf K}(x)$. Every scalar function on $M$ which is invariant under isometries of $(M, F)$ is a constant function. The homogeneity of $(M, F)$ and invariancy of the flag curvature under isometries of $F$ imply that ${\bf K}=constant$. Then, by Theorem \ref{MainTHM} we get the proof.
\qed

\bigskip

\noindent
{\bf Proof of Corollary  \ref{cor2}:} Let $F=\alpha\phi(s)$, $s=\beta/\alpha$, be  a two-dimensional  normal homogeneous generalized Berwald  $(\alpha,\beta)$-metric. In \cite{XDn}, Xu-Deng proved that every normal homogeneous manifold  has vanishing S-curvature and non-negative flag curvature. By Theorem \ref{MainTHM} and the same method used in the proof of Corollary  \ref{cor1}, it follows that $F$ is a Riemannian metric of non-negative constant sectional curvature or a locally Minkowskian metric.
\qed

\bigskip

\noindent
{\bf Proof of Corollary  \ref{cor3}:} Let $(G/H, F)$ be a generalized normal homogeneous Randers manifold. In \cite{ZD}, Zhang-Deng proved that $F$  has vanishing
S-curvature (Corollary 3.11). Also, they showed that any generalized normal homogeneous Randers metric has non-negative flag curvature (see Proposition 3.13 in \cite{ZD}).  Then, by Theorem \ref{MainTHM} we get the proof.
\qed

\bigskip

According to Theorem \ref{MainTHM}, every two-dimensional generalized Berwald $(\alpha, \beta)$-metric with vanishing S-curvature is Riemannian or locally Minkowskian. Every left invariant Finsler metric is a generalized Berwald  metric \cite{A15}. Here, we prove Theorem \ref{MainTHM2} which states that  left invariant Finsler metrics with  vanishing S-curvature  reduce to Riemannian metrics, only. The approach of the  proof  of  Theorem \ref{MainTHM2} is completely different from Theorem \ref{MainTHM}.

\bigskip

\noindent
{\bf Proof of Theorem \ref{MainTHM2}:} To prove a homogeneous surface with ${\bf S}=0$ is Riemannian, we only need to consider the nontrivial case, namely, a 2-dimensional non-Abelian Lie group $G$ with a left invariant Finsler metric $F$. At each $ y\in\mathfrak{g}$ with $F(y)=1$, there is a ${\bf g}_y$ orthonormal basis $e_1=y$ and $e_2$ tangent to $F=1$. At almost all non-zero $y$, the spray vector field $\eta$ is nonzero, i.e.,  $\eta(y)$ is a nonzero multiple of $e_2$. By the homogenous S-curvature formula, we have
\[
{\bf S}(y)=-{\bf I}(\eta)=-{\bf C}_y(\eta,e_2,e_2)=0.
\]
Then, ${\bf C}_y(e_2, e_2, e_2)=0$ everywhere. The speciality of 2-dimensional spaces implies  ${\bf C}=0$ everywhere, so $F$ is Riemannian. By the same method used to prove Corollary  \ref{cor1}, one can conclude that the Riemannian metric is of constant sectional curvature. This completes the proof.
\qed

\bigskip
\noindent
{\bf Proof of Corollary  \ref{cor4}:} By assumption, $F$ has isotropic Berwald curvature
\be
B^i_{\ jkl} = cF^{-1}\Big\{h_{jk}h^i_l+h_{jl}h^i_k+h_{kl}h^i_j+2C_{jkl}y^i\Big\}.\label{IB}
\ee
where $c=c(x)$ is a scalar function on $M$. In \cite{TR}, it is proved that every Finsler surface of  isotropic Berwald curvature \eqref{IB} metric has isotropic S-curvature   ${\bf S}=3cF$. By  Xu-Deng's result in \cite{XDn}, $F$ has vanishing S-curvature ${\bf S}=0$. Then, by Theorem \ref{MainTHM2}, $F$ reduces to a Riemannian metric.
\qed

\section{Some Examples  of Generalized Berwald Manifolds}
In this section, we are going to give some important examples of the class of generalized Berwald manifolds. First, by using trans-Sasakian structure, we construct a family of odd-dimensional generalized Berwald Randers metrics.

\begin{ex} $\big(\textrm{Odd-dimensional generalized Berwald Randers metrics}\big)$
Let $M$ be a differentiable manifold of dimension $2n+1$. Suppose that $\eta=\eta_i(x)dx^i$, $\xi=\xi^i\partial/\partial x^i$ and $\varphi=\varphi^i_j\partial/\partial x^i\otimes dx^j$ are a $1$-form, a vector field, and  a $(1, 1)$-tensor on $M$, respectively.
The triple $(\eta, \xi, \varphi)$ is called an almost contact structure on $M$ if it satisfies
\[
\varphi(\xi)=0,\ \ \ \ \eta(\xi)=1, \ \ \  \varphi^2 =-I+\eta \otimes \xi.
\]
A differentiable manifold of odd dimension $2n+1$ with an almost contact structure is called an almost contact manifold. Let a manifold $M$ with the $(\eta, \xi,\varphi)$ structure admits a Riemannian metric ${\bf g}$ such that
\[
{\bf g}(\varphi X, \varphi Y)={\bf g}(X,Y)-\eta(X)\eta(Y).
\]
Then $M$ is called an almost contact metric structure and ${\bf g}$ is called a compatible metric. In this case,  $(\eta, \xi,\varphi, {\bf g})$ is called almost contact metric structure. An almost contact metric structure $(\eta, \xi,\varphi, {\bf g})$ on $M$ is called a trans-Sasakian structure  if it satisfies
\[
(\nabla_X \varphi)Y=c_1\big\{{\bf g}(X,Y)\xi-\eta(Y)X\big\}+c_2\big\{{\bf g}(\varphi X,Y)\xi -\eta(Y)\varphi X\big\}
\]
for some scalar functions $c_1=c_1(x)$ and $c_2=c_2(x)$  on $M$, where $\nabla$ denotes the Levi-Civita connection of ${\bf g}$.

Now, let $(\eta, \xi,\varphi, {\bf g})$  be a  trans-Sasakian structure on $M$. Define $\alpha, \beta:TM\rightarrow [0, \infty)$ by
\begin{equation}
\forall (x,y)\in TM, \ \ \ \  \alpha(x,y):=\sqrt{{\bf{g}}_x(y,y)}, \,\,\,\,\,\,\,\beta(x,y):=\epsilon\, \eta_x(y),
\end{equation}
where $0<\epsilon<1$ be a constant. Then, for the  Randers metric $F:=\alpha+\beta$, we have
\[
||\beta||_{\alpha}=\epsilon.
\]
It follows that the class of Randers metrics induced by  trans-Sasakian manifolds $(M, \eta, \xi,\varphi, {\bf g})$ are $(2n+1)$-dimensional generalized Berwald metrics on $M$.
\end{ex}

\bigskip

Here, we give a two-dimensional Randers metric $F=\alpha+\beta$ with vanishing S-curvature. We show that if $F$ is a generalized Berwald metric then it reduces to  a Riemannian metric.

\begin{ex}\label{ex7}
Let ${\bf y}= u{\pa / \pa x} + v {\pa / \pa y} \in T_{(x, y)}\mathbb{R}^2$. Consider the Randers metric $F=\alpha+\beta$, where
$\alpha=\alpha({\bf y})$ and $\beta =\beta({\bf y})$ are given by
 \begin{eqnarray*}
{\alpha}  \!\!\!& := &\!\!\! {\sqrt{ \Big ( 1+(1-\e^2)(x^2+y^2)\Big )(u^2+v^2)
 + \Big ( 1+\e^2 + x^2+y^2\Big ) (xv-yu)^2  }
\over \Big ( 1+(1-\e^2)(x^2+y^2) \Big ) \sqrt{ 1+x^2+y^2} },\\
{\beta} \!\!\!& := &\!\!\! - { \e (xv-yu)\over 1+(1-\e^2) (x^2+y^2) },
\end{eqnarray*}
and $\e $ is a real constant (see \cite{Shpf}). $F$ is defined on the whole sphere for $|\e | < 1$. It is remarkable that, one can rewrite $F$ in a polar coordinate system,
$ x = r \cos(\theta)$, $y = r \sin(\theta)$.  Express
\[
\alpha =\sqrt{ a_{11}\mu^2+ a_{12} \mu\nu + a_{21} \nu \mu+ a_{22}\nu^2}, \ \ \ \beta= b_1 \mu +b_2 \nu,
\]
where
\begin{eqnarray*}
&&a_{11}= {1\over (1+r^2)\big(1+(1-\e^2)r^2\big)},\ \  \ a_{12} = a_{21}=0, \  \ \  a_{22} = { r^2 (1+r^2)\over \big(1+(1-\e^2)r^2\big)^2},\\
&&b_1 =0, \ \ \ \ \   \ \ \ b_2 = - {\e r^2 \over 1 + (1-\e^2) r^2 }.
\end{eqnarray*}
Let us put $A:=det(a_{ij})$. Then, we get
\begin{eqnarray*}
&&a^{11}={ r^2 (1+r^2)\over A\big(1+(1-\e^2)r^2\big)^2}, \ \ \ a^{22}={1\over A(1+r^2)\big(1+(1-\e^2)r^2\big)}, \ \ \ \  a^{12}= a^{21}=0,\\
&&b^1=0, \ \ \ \  b^2=-{\e r^2\over A(1+r^2)\big(1 + (1-\e^2) r^2\big)^2}.
\end{eqnarray*}
Therefore, we obtain
\begin{eqnarray}
\|\beta\|_{\alpha}^2=a^{ij}b_ib_j=b_ib^i={\e^2 r^4 \over A(1+r^2)\big(1 + (1-\e^2) r^2\big)^3 }\neq constant.\label{tb}
\end{eqnarray}
\eqref{tb} means that $F$ is not a generalized Berwald metric. On the other hand, a direct computation yields
\begin{eqnarray*}
&&r_{11}= r_{22}=0, \ \ \  \ \ \ r_{12}=r_{21}={ \e^3 r^3\over (1+r^2)\big (1+ (1-\e^2)r^2\big )^2},\\
&&s_{11}=  s_{22}=0, \ \  \ \ \  s_{12}={ \e r \over \big ( 1+ (1-\e^2)r^2\big  )^2 } = - s_{21}\\
&&s_1= { \e^2 r\over (1+r^2) \big ( 1+ (1-\e^2)r^2\big ) }, \ \ \ \ \  \ \ s_2= 0,\\
&& r_1=-{\e^4 r^5\over A(1+r^2)^2\big(1 + (1-\e^2) r^2\big)^4}, \ \  r_2=0.
\end{eqnarray*}
It is easy to find that  $r_{ij}+b_is_j+b_js_i=0$. Then by Lemma 3.2.1 in \cite{CSb}, we get  ${\bf S}=0$. Also, one can see that the following holds
\be
r_i+s_i=\frac{\e^2 r\big[A(1+r^2)( 1+ (1-\e^2)r^2)^3-\e^2 r^4\big]}{A(1+r^2)^2\big(1 + (1-\e^2) r^2\big)^4}.\label{rs}
\ee
According to \eqref{rs},  $F$ is  a generalized Berwald metric (equivalently, $r_i+s_i=0$) if and only if $\e=0$ or the following holds
\be
\big(1+(1-\e^2)r^2\big)^4=0.\label{ep}
\ee
\eqref{ep} contradicts with $|\e | < 1$. Therefore,  $F$ is  a generalized Berwald metric  if and only if $\e=0$  or equivalently $\beta=0$. In this case,  $F$ reduces to the  standard Riemannian metric $F=\alpha$.
\end{ex}

\bigskip

\begin{ex}\label{ex5} (Xu) It is proved that a Finsler metric $F=F(x, y)$ is of Randers type $F=\alpha+\beta$ if and only if it is a solution of the navigation problem on a Riemannian manifold $(M, {\bf h})$ (see \cite{CSb}). Zermelo navigation is an efficient method to study of Randers metrics with certain Riemannian and non-Riemannian curvature properties. More precisely, any Randers metric $F=\alpha+\beta$ on a manifold $M$ is a solution of the following Zermelo navigation problem
\[
\textbf{h}\Big(x,\frac{y}{F}-\mathcal{W}_x\Big)=1,
\]
where $\textbf{h}=\sqrt{h_{ij}(x)y^iy^j}$ is a Riemannian metric and $\mathcal{W}=\mathcal{W}^i(x){\partial}/{\partial x^i}$ is a vector field such that
\[
\textbf{h}(x,-\mathcal{W}_x)=\sqrt{h_{ij}(x)\mathcal{W}^i(x)\mathcal{W}^j(x)}<1.
\]
In fact, $\alpha$ and $\beta$ are given by
\[
\alpha=\frac{\sqrt{\lambda h^2+\mathcal{W}_0}}{\lambda},\qquad \beta=-\frac{\mathcal{W}_0}{\lambda},
\]
respectively and moreover,
\[
\lambda:=1-\|\mathcal{W}\|^2_h, \ \ \ \ \ W_0:=h_{ij}\mathcal{W}^iy^j.
\]
For more details, see   \cite{CSb}. Then, $F$ can be written as follows
\begin{equation}
\label{form of F}
F=\frac{\sqrt{\lambda \textbf{h}^2+\mathcal{W}_0^2}}{\lambda}-\frac{\mathcal{W}_0}{\lambda}.
\end{equation}
In this case, the pair $({\bf h}, \mathcal{W})$  is called the navigation data of $F$.

Now, let $G/H$ be any homogeneous manifold and $\mathfrak{g}=\mathfrak{h}+\mathfrak{m}$
is its reductive decomposition. Suppose $\mathfrak{m}=\mathfrak{m}_0+\mathfrak{m}_1$
be an $\mathrm{Ad}(H)$-invariant decomposition, in which $\mathfrak{m}_0$ is 1-dimensional
and the $\mathrm{Ad}(H)$-action on $\mathfrak{m}_0$ is trivial. Let ${\bf h}$ be a $G$-invariant
Riemannian metric on $G/H$, such that $\mathfrak{m}_0$ and $\mathfrak{m}_1$ are $h$-orthogonal
to each other. Let $\mathcal{W}$ be a $G$-invairant vector field on $G/H$, such that $\mathcal{W}(o)\in\mathfrak{m}_0\backslash\{0\}$.
Then, the navigation process with the data $({\bf h}, \mathcal{W})$ provides a $G$-invariant generalized Berwald Randers metric with ${\bf S}=0$ (see \cite{HD}).
\end{ex}

\begin{ex} \label{ex6} (Xu) As we mentioned in Introduction, the Bao-Shen's Randers metrics on $\mathbb{S}^3$ are concrete generalized Berwald metrics, namely they  are not Berwaldian. Any non-Riemannian homogeneous Randers sphere $\mathbb{S}^3=SU(3)/SU(2)$ (including Bao-Shen's Randers metrics) satisfies ${\bf S}=0$ with constant pointwise $\|\beta\|_{\alpha}$-norms. Then, every non-Riemannian homogeneous Randers sphere is a generalized Berwald metric.  An $\mathbb{S}^3\times \mathbb{S}^1$, in which $\mathbb{S}^3=SU(3)/SU(2)$ and the navigation field is tangent to the $\mathbb{S}^3$-factor, is a 4-dimensional  generalized Berwald Randers metric (see \cite{X}).
\end{ex}

\bigskip

\noindent
{\bf Acknowledgments:} The authors are so  grateful to Ming Xu for his valuable comments on this manuscript.  Likewise, we thank him for providing us with examples \ref{ex5} and \ref{ex6} which improve the quality of our manuscript. Also, we  are thankful to  Behzad Najafi, Mansoor Barzegari  and  Libing Huang  for their reading of this manuscript  and their comments.

\bigskip

\noindent
Akbar Tayebi and Faezeh Eslami\\
Department of Mathematics, Faculty of Science\\
University of Qom \\
Qom. Iran\\
Email:\ akbar.tayebi@gmail.com\\
Email:\ faezeh.eslami70@gmail.com

\begin{thebibliography}{00}
%--------------------------------------------------------------------------------------------------------------------------------------------------
 \bibitem{AZ} H. Akbar-Zadeh, {\it Sur les espaces de Finsler \'{a} courbures sectionnelles constantes}, Bull. Acad. Roy. Bel. Cl, Sci, 5e S\'{e}rie
 - Tome LXXXIV (1988), 281-322.
%-------------------------------------------------------------------------------------------------------------------------------------------------------------------------------
\bibitem{A15} B. Aradi, {\it Left invariant Finsler manifolds are generalized Berwald}, Eur.\ J.\ Pure Appl.\ Math.\ \textbf{8}(1) (2015), 118--125.
%-------------------------------------------------------------------------------------------------------------------------------------------------------------------------------
%\bibitem{ABT} B. Aradi, M. Barzagari and A. Tayebi, {\it Conjugate and conformally conjugate parallelisms on Finsler manifolds}, Periodica. Math. Hungarica, {\bf 74}(2017), 22-30.
%-------------------------------------------------------------------------------------------------------------------------------------------------------------------------------
%\bibitem{AK} B. Aradi and D. Cs. Kert\'{e}sz, {\it A characterization of holonomy invariant functions on tangent bundles}, Balkan. Jour. Geom. Appl. {\bf 19}(2) (2014), 1-10.
%-------------------------------------------------------------------------------------------------------------------------------------------------------------------------------
\bibitem{As} G. S. Asanov, {\it Finsleroid-Finsler space with Berwald and Landsberg conditions}, arXiv:math0603472.
%-------------------------------------------------------------------------------------------------------------------------------------------------------------------------------
\bibitem{BM} N. Bartelme${\ss}$ and  V. S. Matveev, {\it Monochromatic metrics are generalized Berwald}, Differ. Geom. Appl. {\bf 58}(2018), 264-271.
%-------------------------------------------------------------------------------------------------------------------------------------------------------------------------------
\bibitem{CSb} X. Cheng and Z. Shen, {\it Finsler Geometry- An Approach via Randers Spaces}, Springer, Heidelberg and Science Press, Beijing, 2012.
%-------------------------------------------------------------------------------------------------------------------------------------------------------------------------------
\bibitem{ChSh3} X. Cheng and Z. Shen, {\it A class of Finsler metrics with isotropic S-curvature}, Israel J. Math. {\bf 169}(2009), 317-340.
%-------------------------------------------------------------------------------------------------------------------------------------------------------------------------------
%\bibitem{D}  S. Deng, {\it Homogeneous Finsler Spaces},  Springer, New York, 2012.
%-------------------------------------------------------------------------------------------------------------------------------------------------------------------------------
\bibitem{HD} Z. Hu and S. Deng, {\it Homogeneous Randers spaces with isotropic S-curvature and positive flag curvature}, Math. Z. {\bf 270}(2012), 989-1009.
%-------------------------------------------------------------------------------------------------------------------------------------------------------------------------------
\bibitem{HM}   L. Huang and  X. Mo, {\it Geodesics of invariant Finsler metrics on a Lie group}, preprint.
%-------------------------------------------------------------------------------------------------------------------------------------------------------------------------------
\bibitem{I} S. Ivanov, {\it Monochromatic Finsler surfaces and a local ellipsoid characterisation}, Proc. Am. Math. Soc. {\bf 146}(2018), 1741-1755.
%-------------------------------------------------------------------------------------------------------------------------------------------------------------------------------
\bibitem{MW} X. Mo and X. Wang, {\it On Finsler metrics of constant S-curvature}, Bull. Korean Math. Soc. {\bf 50}(2) (2013),  639-648.
%-------------------------------------------------------------------------------------------------------------------------------------------------------------------------------
%\bibitem{CSY} X. Cheng, Z. Shen and  G. Yang, {\it On a class of two-dimensional Finsler manifolds of isotropic S-curvature}, Sci. China Math. {\bf 61}(2018),  57-72.
%-------------------------------------------------------------------------------------------------------------------------------------------------------------------------------
\bibitem{CWW}  X. Cheng, H. Wang and M. Wang,  {\it $(\alpha,\beta)$-metrics with  relatively isotropic mean Landsberg curvature}, Publ. Math. Debrecen. {\bf 72}(2008), 475-485.
%-------------------------------------------------------------------------------------------------------------------------------------------------------------------------------
\bibitem{ShC} Z. Shen, {\it On a class of Landsberg metrics in Finsler geometry}, Canadian. J. Math. {\bf 61}(6) (2009), 1357-1374.
%------------------------------------------------------------------------------------------------------------------
\bibitem{ShDiff} Z. Shen, {\it Differential Geometry of Spray and Finsler Spaces}, Kluwer Academic Publishers,  2001.
%-------------------------------------------------------------------------------------------------------------------------------------------------------------------------------
\bibitem{ShSK} Z. Shen, {\it Finsler metrics with ${\bf K}=0$ and ${\bf S}=0$}, Canadian J. of Math. {\bf 55}(2003),  112-132.
%-------------------------------------------------------------------------------------------------------------------------------------------------------------------------------
\bibitem{Shpf} Z. Shen, {\it Two-dimensional Finsler metrics of constant curvature}, Manuscripta Mathematica. {\bf 109}(3) (2002), 349-366.
%-----------------------------------------------------------------------------------
%\bibitem{Shvol} Z. Shen, {\it  Volume comparison and its applications in Riemann-Finsler geometry}, Adv. Math. {\bf 128}(1997), 306-328.
%-------------------------------------------------------------------------------------------------------------------------------------------------------------------------------
%\bibitem{Y}   G. Yang, {\it A note on a class of Finsler metrics of isotropic S-curvature}, arXiv:1310.3463v2.
%-------------------------------------------------------------------------------------------------------------------------------------------------------------------------------
%\bibitem{Szabo1} Z. I. Szab\'{o}, {\it Positive definite Berwald spaces (structure theorems on Berwald spaces)}, Tensor (N.S.), {\bf 35} (1981), 25-39.
%-------------------------------------------------------------------------------------------------------------------------------------------------------------------------------
\bibitem{Szabo} Z. I. Szab\'{o}, {\it Generalized spaces with many isometries}, Geometria Dedicata, {\bf 11}(1981), 369-383.
%-------------------------------------------------------------------------------------------------------------------------------------------------------------------------------
\bibitem{SzSz1} Sz.~Szak\'al and J.~Szilasi, {\it A new approach to generalized Berwald manifolds \relax{I}}, SUT J. Math.\ \textbf{37}(2001),  19--41.
%-------------------------------------------------------------------------------------------------------------------------------------------------------------------------------
\bibitem{SzSz2} J.~Szilasi and Sz.~Szak\'al, {\it A new approach to generalized Berwald manifolds \relax{II}}, Publ.~Math.~Debrecen,  \textbf{60}(2002), 429--453.
%-------------------------------------------------------------------------------------------------------------------------------------------------------------------------------
%\bibitem{SzLK11} J.~Szilasi, R.~L.~Lovas and D.~\relax{Cs}.~Kert\'esz, {\it Several ways to a Berwald manifold -- and some steps beyond}, Extracta Math.\ \textbf{26}(2011), 89--130.
%-------------------------------------------------------------------------------------------------------------------------------------------------------------------------------
\bibitem{BSZI} J. Szilasi, R. L. Lovas, and D. Cs. Kert\'{e}sz, {\it Connections, Sprays and Finsler structures}, World Scientific, 2014.
%-------------------------------------------------------------------------------------------------------------------------------------------------------------------------------
\bibitem{TB} A. Tayebi and M. Barzegari, {\it Generalized Berwald spaces with  $(\alpha, \beta)$-metrics}, Indagationes. Math. (N.S.). {\bf 27}(2016), 670-683.
%-------------------------------------------------------------------------------------------------------------------------------------------------------------------------------
\bibitem{TP} A. Tayebi and E. Peyghan, {\it  On Douglas surfaces}, Bull. Math. Soc. Science. Math. Roumanie. Tome {\bf 55} (103), No 3,  (2012), 327-335.
%-------------------------------------------------------------------------------------------------------------------------------------------------------------------------------
\bibitem{TR} A. Tayebi and   M. Rafie Rad, {\it S-curvature of isotropic Berwald metrics}, Sci. China. Series A: Math.  {\bf 51}(2008), 2198-2204.
%-------------------------------------------------------------------------------------------------------------------------------------------------------------------------------
\bibitem{XDg} M. Xu and S. Deng, {\it Killing frames and S-curvature of homogeneous Finsler spaces}, Glasgow. Math. Journal. {\bf 57}(2015), 457-464.
%-------------------------------------------------------------------------------------------------------------------------------------------------------------------------------
\bibitem{XDn} M. Xu and S. Deng, {\it Normal homogeneous Finsler spaces}, Transformation Groups. {\bf 22}(2017), 1143-1183.
%-------------------------------------------------------------------------------------------------------------------------------------------------------------------------------
\bibitem{Vin0} C. Vincze, {\it On a special type of generalized Berwald manifolds: semi-symmetric linear connections preserving the Finslerian length of tangent vectors}, European Journal of Math. {\bf 3}(2017), 1098-1171.
%-------------------------------------------------------------------------------------------------------------------------------------------------------------------------------
%\bibitem{Vin1} C. Vincze, {\it On geometric vector fields of Minkowski spaces and their applications}, Differ. Geom. Appl. {\bf 24}(1) (2006),  1-20.
%-------------------------------------------------------------------------------------------------------------------------------------------------------------------------------
\bibitem{V1} C. Vincze, {\it On Randers manifolds with semi-symmetric compatible linear connections}, Indagationes. Math. (N.S.). {\bf 26}(2015), 363-379.
%-------------------------------------------------------------------------------------------------------------------------------------------------------------------------------
\bibitem{Vin2} C. Vincze, {\it On generalized Berwald manifolds with semi-symmetric compatible linear connections}, Publ. Math. Debrecen. {\bf 83}(2013),  741-755.
%-------------------------------------------------------------------------------------------------------------------------------------------------------------------------------
\bibitem{V3} C. Vincze, T. R. Khoshdani, S. M. Z. Gilani, and M. Ol\'{a}h, {\it  On compatible linear connections
of two-dimensional generalized Berwald manifolds: a classical approach}, Commun. Math. {\bf 27}(2019), 51-68.
%-------------------------------------------------------------------------------------------------------------------------------------------------------------------------------
%\bibitem{V2} C. Vincze, T. R. Khoshdani, and M. Ol\'{a}h, , {\it  On generalized Berwald surfaces with locally symmetric fourth root metrics}, Balkan J. Geom. Appl., {\bf 24}(2019), 63-78.
%-------------------------------------------------------------------------------------------------------------------------------------------------------------------------------
\bibitem{X} M. Xu, {\it Geodesic orbit spheres and constant curvature in Finsler geometry}, Differ. Geom. Appl. {\bf 61}(2018), 197-206.
%-------------------------------------------------------------------------------------------------------------------------------------------------------------------------------
\bibitem{ZD} L. Zhang and S. Deng, {\it On generalized normal homogeneous Randers spaces}, Publ. Math. Debrecen. {\bf 90}(2017), 507-523.
%-------------------------------------------------------------------------------------------------------------------------------------------------------------------------------
\end{thebibliography}
\end{document}